\documentclass[11pt]{amsart}
\usepackage{amsmath, amsthm, amssymb,amscd}
\usepackage{float}
\usepackage{graphicx}
\usepackage{xypic}
\usepackage[arrow, matrix, curve]{xy}
\usepackage{color}
\usepackage{enumitem}  
\usepackage{tikz-cd} 
\usepackage{url}

\usepackage{blindtext}

\usepackage[english]{babel}
\newcommand{\seq}{\subseteq}
\newcommand{\C}{\mathbb{C}}

\newtheorem{thm}{Theorem}[section]
\newtheorem*{thm-nl}{Theorem}
\newtheorem*{prop-nl}{Proposition}

\newtheorem{lem}[thm]{Lemma}

\def\PP{{\textbf P}}

\def\K{\mathcal{K}}

\def\Pic0{{\rm Pic}^0(X)}

\newtheorem{cor}[thm]{Corollary}
\newtheorem*{cor-nl}{Corollary}
\newtheorem{conjecture}[thm]{Conjecture}
\newtheorem*{conjecture-nl}{Conjecture}
\newtheorem{defin}[thm]{Definition}
\newtheorem*{quest-nl}{Question}
\newtheorem*{quests-nl}{Questions}

\newtheorem{prop}[thm]{Proposition}

\theoremstyle{remark}

\title{{Projecting Syzygies of Curves}}
\date{\today}

\author[M. Kemeny]{Michael Kemeny}

\address{University of Wisconsin-Madison, Department of Mathematics, 480 Lincoln Dr
\hfill \newline\texttt{}
 \indent WI 53706, USA} \email{{\tt michael.kemeny@gmail.com}}

\begin{document}
\begin{abstract}
We explore the concept of projections of syzygies and prove two new technical results; we firstly give a precise characterization of syzygy schemes in terms of their projections, secondly, we prove a converse to Aprodu's Projection Theorem. Applying these results, we prove a conjecture of \cite{lin-syz} stating that extremal syzygies of general curves of non-maximal gonality embedded by a linear system of sufficiently high degree arise from scrolls. Lastly, we prove Green's Conjecture for general covers of elliptic curves (of arbitrary degree) as well as proving a new result for curves of even genus and maximal gonality.

\end{abstract}
\maketitle
\setcounter{section}{-1}
\section{Introduction}
Let $X$ be a projective variety and $L$ a line bundle, assumed to be very ample for simplicity. One of the most fundamental objects in algebraic geometry is the section ring $$\Gamma_X(L):=\bigoplus_{n \in \mathbb{Z}} H^0(X,L^{\otimes n}).$$ In order to understand the structure of $\Gamma_X(L)$, one treats it as a $\text{Sym} \left(H^0(X,L) \right)$ module and takes the minimal free resolution. The \emph{syzygy spaces} $K_{i,j}(X,L)$ are then the graded pieces which appear in the resolution. Their dimensions give important invariants $b_{i,j}(X,L)$ of the polarized variety $(X,L)$, known as \emph{Betti numbers}.\medskip

Our understanding of the syzygies of $(X,L)$ is highly limited. At present, our knowledge is most complete in the case where $\dim X=1$, i.e.\ $X$ is an algebraic curve $C$. Most of the known results are \emph{vanishing theorems}, providing conditions for the vanishing of Betti numbers $b_{i,j}(C,L)$. Some highlights include Voisin's theorem on Generic Green's Conjecture, which concerns the case $L=\omega_C$, \cite{V1}, \cite{V2}, as well as the recent proofs of the Gonality Conjecture \cite{ein-lazarsfeld} and the Generic Secant Conjecture \cite{generic-secant}, which are concerned with the case of more general line bundles $L$. Whilst the Gonality Conjecture is known for \emph{arbitrary} curves, the full statement for Green's Conjecture and the Secant Conjecture are still very much open in the non-generic case.

\medskip

More recently, work has been done on trying to go beyond vanishing theorems and describe certain \emph{values} of $b_{i,j}(C,L)$ in terms of geometry. For example, the first linear Betti number $b_{1,1}(C,L)$ describes the number of quadrics required to generate the ideal of $I_C$. It is not unreasonable to hope that other Betti numbers likewise carry geometric information. In \cite{lin-syz} and \cite{betti-multiple} attention was focused on the last of the nonzero, linear Betti numbers $b_{g-k,1}(C,\omega_C)$ for canonical curves of non-maximal gonality $k$. It was found that this extremal Betti numbers counts a classically studied invariant, namely the number of \emph{minimal pencils}, i.e.\ maps $f: C \to \PP^1$ of degree equal to the gonality $k$ (up to genericity hypotheses).\medskip

 The number of minimal pencils on a general $k$-gonal curve was first showed to be finite by Segre, \cite{segre}. More recently Arbarello--Cornalba and Coppens have obtained several important results on the loci of curves carrying multiple minimal pencils, see \cite{arbarello-cornalba}, \cite{copp2}, \cite{coppens}. In particular, a general curve of nonmaximal gonality has a \emph{unique} minimal pencil and in this case Schreyer has conjectured $b_{g-k,1}(C,\omega_C)=g-k$, which both implies Green's Conjecture $b_{g-k+1,1}(C,\omega_C)=0$ and gives a more geometric interpretation for it, as we explain below. \medskip

One of the goals of this paper is to prove a result interpolating and extending both the vanishing theorem of \cite{ein-lazarsfeld} and the verification of Schreyer's conjecture in \cite{lin-syz}. 
\begin{thm} \label{general-schreyer}
Let $C$ be a general $k$-gonal curve $C$ of genus $g \geq 2k-1$ for $k \geq 4$. Let $L$ be an arbitrary line bundle on $C$ of degree $\deg(L) \geq 2g+k$. Then $$b_{r(L)-k,1}(C,L)=r(L)-k.$$
\end{thm}
In the result above, $r(L)=h^0(L)-1$. This result was conjectured in \cite[Conjecture 0.5]{lin-syz}. \medskip

To explain the importance of Theorem \ref{general-schreyer} and illustrate the link to \cite{ein-lazarsfeld}, we first reinterpret the statement geometrically. One of the only classes of varieties for which it is possible to determine the minimal free resolution are \emph{determinantal varieties} described by the degeneracy locus of a morphism of vector bundles $\mathcal{V}_1 \to \mathcal{V}_2$, see \cite{lascoux}. One strategy to study syzygies of a curve $C$ is to embed it into a determinantal variety $Z$ and then restrict the determinantal syzygies to $C$. \medskip

This strategy was employed in Voisin's proof of Green's conjecture for an even genus $g=2k$ curves $C$ lying on a $K3$ surface $X$. In this case the determinantal variety is a Grassmannian $G(k+2,2)$ produced out of a rank two Lazarsfeld--Mukai bundle on $X$ (itself resulting from a minimal pencil on $C$). Voisin then proves that the length of the linear strand of the resolution of $C$ is equal to that of $G(k+2,2)$. See also the recent paper \cite{AFPRW}, which removes the need to study Hilbert schemes as in Voisin's original approach. \medskip

For another example, a minimal pencil $f: C \to \PP^1$ of degree $k$ induces a scroll 
$$X_L:=\bigcup_{p \in \PP^1} \langle f^{-1}(p) \rangle \seq \PP^{r(L)} .$$ The scroll is a determinantal variety containing the embedded curve $(C,L)$. In this case the Lascoux resolution of the scroll simplifies to an explicit resolution first found by Eagon--Northcott, see \cite[(6.1.6)]{weyman-book}. By comparing the syzygies of $(C,\omega_C)$ to those of the scroll $X_{\omega_C}$, Schreyer has classified all Betti tables of canonical curves of genus $g \leq 9$ in geometric terms, \cite{schreyer1}. \medskip

If $\deg(L) \geq 2g+k$, then the Eagon--Northcott complex shows $K_{p,1}(X_L,\mathcal{O}(1))=0$ for $p>r(L)-k$, whereas $b_{r(L)-k,1}(X_L,\mathcal{O}(1))=r(L)-k.$ Thus Theorem \ref{general-schreyer} can be interpreted as saying that all linear ${r(L)-k}^{th}$ syzygies of the curve arise from the scroll, i.e.\ restriction induces an isomorphism
$K_{r(L)-k,1}(X_L,\mathcal{O}(1)) \simeq K_{r(L)-k,1}(C,L).$ Since there are no relations amongst the ${r(L)-k}^{th} $ linear syzygies of the scroll $X_L$ the isomorphism above implies the vanishing $b_{p,1}(C,L)=0$ for $p>r(L)-k$ as predicted by \cite{ein-lazarsfeld}. In this way Theorem \ref{general-schreyer} provides a more geometric explanation for this vanishing statement, which was originally proved for line bundles of asymptotically high degree using Serre vanishing.\medskip

Theorem \ref{general-schreyer} also provides the first non-trivial step towards understanding the non-zero terms in the minimal free resolution of $C$, by giving the value of the last nonzero Betti number on the linear strand. This extremal Betti number is known to be responsible for much of the variance in the Betti tables of curves, \cite{schreyer-topics}. \medskip

The bound of Theorem \ref{general-schreyer} is optimal. Suppose $g=2k-1$ and $C$ is a general curve of genus $g$ and gonality $k$. Then, $C$ admits both $A \in W^1_k(C)$ as well as $B \in W^1_{k+1}(C)$. Set $L:=\omega_C\otimes B$. There are \emph{two} scrolls containing $C$, the scroll $X_{L}$ given by the union of $\text{Span}(D)$ for $D \in |A|$ as well as the scroll $Y_L=\bigcup_{E \in |B|} \text{Span}(E)$. Both scrolls contribute syzygies, implying $b_{r(L)-k,1}(C,L)>r(L)-k$. \medskip

The second goal of this paper is to use geometric constructions to refine the existing results on Green's Conjecture for \emph{special} canonical curves, \cite{green-koszul}. Firstly, we prove that Green's conjecture holds for general covers of elliptic curves of arbitrary degree, generalizing a result of Aprodu--Farkas when $d=3$ \cite{AF-covers}.
\begin{thm} \label{intro-elliptic}
Fix any elliptic curve $E$ and let $f: C \to E$ be a general, degree $d$ primitive cover of $E$ for $d \geq 3$. Then Green's Conjecture holds for $C$.
\end{thm}

Our interest in this result derives from several angles. On the one hand, the strongest known prior result on Green's Conjecture is Aprodu's Theorem, \cite{aprodu-remarks}, stating that Green's conjecture holds for any $k$-gonal curve $C$ of genus $g$ provided one has the \emph{linear-growth condition}
$$\dim W^1_{k+n}(C) \leq n, \; \; \text{for all $0 \leq n \leq g+2-2k$} .$$
 Aprodu's Theorem was a crucial part of Aprodu--Farkas' well-known result that Green's conjecture holds for curves on arbitrary K3 surfaces, \cite{aprodu-farkas}.\medskip

The linear growth condition is conjecturally equivalent to requiring $$\dim W^1_{k}(C)=0,$$ provided $g \leq 2k-1$. Hence the most interesting open case of Green's Conjecture is for curves where this condition fails, i.e.\ for curves of non-maximal gonality with infinitely many minimal pencils. Perhaps the most natural example of such curves arises from general covers $f: C \to E$ of elliptic curves of degree $d$. Provided $2d \leq \lfloor \frac{g+3}{2} \rfloor$, then such curves $C$ have gonality $2d$ and note that we always have $\dim W^1_{2d}(C) \geq 1$, by pulling back line bundles of degree two from $E$.\medskip

Covers of elliptic curves also play an important role in the computer experiments of Schreyer, \cite[\S 6]{schreyer-topics}, which indicate that the Betti tables of covers of elliptic curves seem to behave very differently from general curves of gonality $2d$. In particular, the last Betti number $b_{g-2d,1}(C,\omega_C)$ in the $2$-linear strand need not be a multiple of $g-k$, which is rather strange given the computations of Hirschowitz--Ramanan \cite{hirsch}. It is this phenomenon that led us to consider such curves.\medskip

The last result we prove is an improvement to the prior results on Green's conjecture for curves of even genus and maximal gonality. Recall that results of Hirschowitz--Ramanan \cite{hirsch}, when combined with \cite{V2}, establish Green's Conjecture for all curves of odd genus and maximal gonality. In the case of curves of even genus $g=2n$ and maximal gonality $k=n+1$, the best known prior result is Aprodu's Theorem, establishing the conjecture for curves $C$ with $\dim W^1_{n+1}(C)=0$. One always has $\dim W^1_{n+1}(C)\leq 1$ by \cite{FHL}, so it remains to analyse the case $$\dim W^1_{n+1}(C)= 1.$$ We go quite a long way towards resolving this remaining case:
\begin{thm} \label{green-even}
Let $C$ be a smooth curve of genus $g=2n$ and gonality $k=n+1$. Suppose for $x,y \in C$ general, there is at most one $A \in W^1_{n+1}(C)$ such that $A(-x-y)$ is effective. Further assume $h^0(C,A^{\otimes 2})=3$. Then $C$ satisfies Green's Conjecture.
\end{thm}
An example of a curve satisfying the assumptions above is given by a general degree $d$ primitive cover $f: C \to E$ of an elliptic curve, where $n$ is odd, $2d=n+1$ and $C$ has genus $2n=4d-2$. \medskip

\subsection{Techniques} The results above are proven by associating geometric varieties and constructions to syzygy spaces. In the process, we obtain some technical results on syzygies which we believe are of independent interest. \medskip

To explain the techniques, first recall that syzygy spaces $K_{i,j}(X,L)$ can be identified with the middle cohomology of
$$\bigwedge^{i+1}H^0(L) \otimes H^0((j-1)L) \xrightarrow{\delta} \bigwedge^{i}H^0(L) \otimes H^0(jL) \xrightarrow{\delta} \bigwedge^{i-1}H^0(L) \otimes H^0((j+1)L) ,$$
where $\delta$ is the Koszul differential.\medskip

Suppose $V \seq H^0(X,L)$ is a codimension one subspace, and let $pr_V : \PP(H^0(X,L)^{\vee}) \dashrightarrow \PP(V^{\vee})$ be the associated projection morphism. As first observed by Green \cite[\S 1.b]{green-koszul-II}, it can be useful to compare the syzygies of $\Gamma_X(L)$ when considered as a $\text{Sym}(V)$ module to its syzygies as a $\text{Sym} \left(H^0(X,L) \right)$ module. Namely, letting $K_{i,j}(\Gamma_X(L),V)$ denote the middle cohomology of 
$$\bigwedge^{i+1}V \otimes H^0((j-1)L) \xrightarrow{\delta} \bigwedge^{i}V \otimes H^0(jL) \xrightarrow{\delta} \bigwedge^{i-1}V \otimes H^0((j+1)L) ,$$
we have the \emph{projection map}
$$pr_V \; : \; K_{i,j}(X,L) \to K_{i-1,j}(\Gamma_X(L),V),$$
see \cite[\S 2.2.1]{aprodu-nagel}. If we suppose that $\phi_L \; : X \hookrightarrow \PP(H^0(X,L)^{\vee})$ is projectively normal, that $x \in X \seq \PP(H^0(X,L)^{\vee})$ is a point corresponding to a codimension one subspace $V \seq H^0(X,L)$, and that further the projected variety $pr_x(X) \seq \PP(V^{\vee})$ is projectively normal, then the projection map factors through a map
$$pr_x \; : \; K_{i,j}(X,L) \to K_{i-1,j}(pr_x(X),\mathcal{O}_{pr_x(X)}(1)),$$
where we have an inclusion $K_{i-1,j}(pr_x(X),\mathcal{O}_{pr_x(X)}(1)) \seq K_{i-1,j}(\Gamma_X(L),V)$.\medskip

The projection map $pr_x \; : \; K_{i,j}(X,L) \to K_{i-1,j}(pr_x(X),\mathcal{O}_{pr_x(X)}(1))$ has been used by Choi, Kang and Kwak to prove that if the projective variety $X \seq \PP(H^0(X,L)^{\vee})$ as above satisfies the Green--Lazarsfeld property $(N_p)$ and if, furthermore, $pr_x(X) \simeq X$, then $pr_x(X) \seq \PP(V^{\vee})$ satisfies property $(N_{p-1})$, \cite{choi-kang-kwak}. Specialising to the case where $X$ is a smooth curve $C$, Aprodu has used the projection map to study the Green and Green--Lazarsfeld Conjectures for curves, \cite{aprodu-higher}. \medskip

Our first technical contribution is to relate the projection map to Green's notion of \emph{syzygy scheme}. For a syzygy $\alpha \neq 0 \in K_{p,1}(X,L)$, the syzygy scheme $\text{Syz}(\alpha)$ is defined to be the largest variety $Y \seq \PP(H^0(X,L)^{\vee})$ containing $X$ such that $\alpha$ arises via restriction from a syzygy $\alpha' \in K_{p,1}(Y,\mathcal{O}_Y(1))$, \cite{green-canonical}. It is well known that one has the following containment
$$pr_x(\text{Syz}(\alpha)) \seq \text{Syz}(pr_x(\alpha))$$
for any $x \in X$, relating the syzygy scheme of a syzygy to the syzygy scheme of its projection, see \cite[\S 3]{aprodu-nagel}. Furthermore, it is known that if $Z \seq X$ is a spanning set of $\PP(H^0(X,L)^{\vee})$ and $\alpha \neq 0$, then there is some $x \in Z$ such that $pr_x(\alpha) \neq 0$, or, equivalently, $\text{Syz}(pr_x(\alpha))\neq \PP(V^{\vee})$, \cite[Prop.\ 2.14]{aprodu-nagel}.  \medskip

We provide here a generalization of both of the above statements, showing that one can completely recover the syzygy scheme of an element $\alpha \in K_{p,1}(X,L)$ from the syzygy scheme of projections $pr_x(\alpha)$. Let 
$X \seq \PP(H^0(X,L)^{\vee})$ be an integral, projective variety embedded by a very ample line bundle $L$ as above, and assume for simplicity that $X$ is projectively normal. For any $x \in X \seq \PP(H^0(X,L)^{\vee})$, let $W_x \seq H^0(X,L)$ be the corresponding codimension one subspace.
\begin{thm}  \label{intro-syz-scheme}
Let $X$ be as above and let $Z \seq X$ be a subset spanning $\PP(H^0(X,L)^{\vee})$. Assume that for all $x \in Z$, $pr_x(X) \seq \PP(W_x^{\vee})$ is linearly normal and non-degenerate. Let $\alpha \neq 0 \in K_{p,1}(X,L)$. Then
$$\text{Syz}(\alpha) = \bigcap_{x \in Z} \text{Cone}_x\left( \text{Syz}(pr_x(\alpha)) \right),$$
where, for any $Y \in \PP(W_x^{\vee})$, $Cone_x(Y) \seq \PP(V^{\vee})$ denotes the cone with vertex $x$.
\end{thm}

Our next technical result is related to the following early application of the technique of projection of syzygies:
\begin{thm} [Aprodu's Projection Theorem] \label{AP1}
Let $C$ be a smooth curve of genus $g$ and suppose $x, y \in C$ are distinct points. Let $D$ be the $g+1$ nodal curve obtained by identifying $x$ and $y$. Suppose $K_{p,1}(C,\omega_C)=0$. Then $K_{p+1,1}(D,\omega_D)=0$.
\end{thm}
This allows one to prove the generic Green's conjecture for curves of a fixed gonality by induction on the genus. Perhaps the most interesting case of Theorem \ref{AP1} is when $C$ is a curve of gonality $k$ and Clifford index $k-2$, with minimal pencil $f: C \to \PP^1$ of degree $k$, and where $p=g-k+1$ and $x,y \in C$ are distinct points with $f(x)=f(y)$. In this case, Theorem \ref{AP1} implies that if $C$ satisfies Green's Conjecture, then the nodal curve $D$ obtained by identifying $x$ and $y$ also has gonality $k$ and Clifford index $k-2$, and furthermore $D$ satisfies Green's conjecture. This provides an approach to proving Green's conjecture for general curves of a fixed gonality using induction and was one of Aprodu's main motivations for formulating Theorem \ref{AP1}. \medskip

It is natural to ask for a converse of this result, i.e.\ to find an assumption on $D$ as above to guarantee that $C$ satisfies Green's Conjecture. We recall from \cite{SSW}, \cite{lin-syz}, \cite{bopp-schreyer}, that Schreyer has stated the following strengthening of Green's Conjecture:
\begin{conjecture}[Schreyer's Conjecture]\label{schrconj1}
Let $C$ be a curve of genus $g$ and non-maximal gonality $3\leq k \leq \frac{g+1}{2}$. Assume $W^1_k(C)=\{A\}$ is a reduced single point and $A$ is the unique line bundle of degree at
most $g-1$ achieving the Clifford index. Then
$$b_{g-k,1}(C,\omega_C)=g-k .$$
\end{conjecture}
Note that the condition that $W^1_k(C)$ is reduced is equivalent to demanding $h^0(C,A^{\otimes 2})=3$.\medskip

Schreyer's conjecture has been proven under the ``bpf-linear growth'' genericity assumption in \cite{lin-syz}. As explained earlier, Schreyer's conjecture implies both Green's Conjecture $b_{g-k+1,1}(C,\omega_C)=0 $ and, further, implies that all syzygies at the end of the $2$-linear strand of the canonical curve come from the scroll $X_{\omega_C}$. In terms of syzygy schemes, Schreyer's conjecture is equivalent to the statement that $\text{Syz}(\alpha)=X_{\omega_C}$ for all $\alpha \neq 0 \in K_{g-k,1}(C,\omega_C)$. \medskip

Our partial converse to the Aprodu Projection Theorem then reads:
\begin{thm} \label{intro-converse}
Let $D$ be the $1$-nodal, $k$-gonal curve as above, with normalization the smooth curve $C$ of genus $g$ and line bundle $B \in W^1_k(D)$ satisfying $\nu^*B \simeq A$.
Assume \begin{enumerate}
\item $h^0(D,B^{\otimes 2})=3$, and
\item  $b_{g+1-k,1}(D,\omega_D)=g+1-k$. 
\end{enumerate}
Then $K_{g+1-k,1}(C,\omega_C)=0$.
\end{thm}
In other words, if the $k$-gonal nodal curve $D$ of genus $g+1$ satisfies Schreyer's Conjecture, then the smooth $k$-gonal curve $C$ of genus $g$ satisfies Green's Conjecture.\medskip

We now briefly explain how these technical results imply the main results. The key ingredients in the proof of Theorem \ref{general-schreyer} are Theorem \ref{intro-syz-scheme} together with the following important result of Eisenbud--Popescu \cite{eisenbud-popescu}:
\begin{thm} [Eisenbud--Popescu] 
Let $X \seq \PP^{r}$ be a rational normal scroll of degree $f$ and $0 \neq \alpha \in K_{f-1,1}(X,\mathcal{O}_X(1))$. Then $\text{Syz}(\alpha) =X$.
\end{thm}
Theorem \ref{intro-elliptic} on Green's conjecture for elliptic curves is proving by combining the Aprodu Projection Theorem with an analysis of the Brill--Noether theory of covers $f: C \to E$ as above, along the lines of \cite{arbarello-cornalba}. Lastly, Theorem \ref{green-even} on Green's conjecture
for curves of even genus is an immediate corollary of Theorem \ref{intro-converse}. \medskip

\noindent {\bf Acknowledgments:} We thank Juliette Bruce for an interesting discussion on projecting syzygies. We thank Gavril Farkas for suggesting that our results could be applied to Green's Conjecture for curves of even genus and maximal gonality. The author was supported by NSF grant DMS-1701245.

\section{Projections of Syzygy Schemes}
The goal of this section is to prove a precise relationship between the syzygy scheme of a linear syzygy and that of its projection. We first need to recall the notion of a syzygy scheme. Let $X \seq \PP^n$ be a closed subscheme such that $X$ is non-degenerate, i.e.\ the restriction map
$$V:=H^0(\PP^n, \mathcal{O}_{\PP^n}(1)) \to H^0(X,\mathcal{O}_X(1))$$
is injective. We let $S(X)$ denote the homogeneous coordinate ring and let $K_{p,q}(S(X),V)$ denote the syzygies of $S(X)$ as a $\text{Sym}(V)=S(\PP^n)$ module. We have an isomorphism $$K_{p,1}(S(X),V) \simeq K_{p-1,2}(I_X,V),$$
where $I_X$ is the ideal of $X$, by \cite{aprodu-nagel}, Prop.\ 1.27, and further
$$K_{p-1,2}(I_X,V) = \text{Ker}(\bigwedge^{p-1}V \otimes (I_X)_2 \to \bigwedge^{p-2}V \otimes (I_X)_3).$$
If $Y \seq \PP^n$ is any closed subscheme containing $X$, then the inclusion $I_Y \seq I_X$ induces an inclusion
$$K_{p,1}(S(Y),V) \simeq K_{p-1,2}(I_Y,V) \xrightarrow{\text{res}^Y_X} K_{p-1,2}(I_X,V) \simeq K_{p,1}(S(X),V).$$

\begin{defin} [\cite{aprodu-nagel}, \S 3]
Let $\alpha$ be a nonzero element of $K_{p,1}(S(X),V)$. Then the syzygy scheme $\text{Syz}(\alpha) \seq \PP^n$ is defined to be the largest closed subscheme $Y \seq \PP^n$ containing $X$ such that $\alpha \in \text{Im}(\text{res}^Y_X)$.
\end{defin}

We recall some basic facts about syzygy schemes from \cite[\S 3]{aprodu-nagel}. Let $V$ be a vector space of dimension $n+1$ and identify $\PP^n \simeq \PP(V^{\vee}):=\text{Proj}(\text{Sym}(V))$. For any $x \in V^{\vee}$, let $W_x \seq V$ denote the kernel of $x: V \to \C$ and let $i_x: \bigwedge^p V \to \bigwedge^{p-1}W_x$ be the contraction mapping defined by
$$i_x(v_1 \wedge \ldots \wedge v_p) = \sum_i (-1)^i v_1 \wedge \ldots \wedge \hat{v_i} \wedge \ldots v_p \otimes x(v_i).$$
The following statement is \cite{aprodu-nagel}, Lemma 3.7 and Prop.\ 3.15.
\begin{prop} \label{AN-proj-syz}
$X \seq \PP(V^{\vee})$ be a non-degenerate, linearly-normal, projective variety.
\begin{enumerate}[label=(\roman*)]
\item Let $\alpha \in K_{p,1}(S(X),V)$ and let $\bar{\alpha} \in \text{Hom}(V^{\vee},\bigwedge^pV) \simeq \bigwedge^pV \otimes V$ be a representative for $\alpha$. Then
$Syz(\alpha)= \{ [x] \in \PP(V^{\vee}) \; | \; i_x(\bar{\alpha}(x))=0 \}$.

\item For any $x \in V^{\vee}$ let $pr_x: \PP(V^{\vee}) \to \PP(W_x^{\vee})$ denote the projection with center $x$. Suppose that both $X$ and $Y:=pr_x(X)$ are linearly normal as well as non-degenerate. Then for any nonzero $\alpha \in K_{p,1}(S(X),V)$, $x \in \text{Syz}(\alpha)$ if and only if there exists $\beta  :  W_x^{\vee} \to \bigwedge^{p-1}W_x$ such that we have a commutative diagram $$\begin{tikzcd}
V^{\vee} \arrow[r, "\bar{\alpha}"] \arrow[d]  & \bigwedge^p V \arrow[d, "i_x"]  \\
W_x^{\vee} \arrow[r,"\beta"] & \bigwedge^{p-1}W_x
\end{tikzcd}.$$

\end{enumerate}
\end{prop}
Part $(ii)$ of the above proposition can be rephrased. In the notation of the proposition, assume $X$ and $Y$ are linearly normal and non-degenerate. Recall from \cite[\S 2.2.1]{aprodu-nagel} that there is a ``projection map'' $$pr_x: K_{p,1}(S(X),V) \to K_{p-1,1}(S(X),W_x).$$ Then one can rephrase Proposition \ref{AN-proj-syz} $(ii)$ as $x \in \text{Syz}(\alpha)$ if and only if $pr_x(\alpha) \in K_{p-1,1}(S(Y),W_x)$.

The next result provides an improvement of \cite[Lemma 3.17(ii)]{aprodu-nagel}.
\begin{thm} \label{char-syz-proj}
Let $X \seq \PP(V^{\vee})$ be an integral, projective variety and let $Z \seq X$ be a set such that $\text{Span}(Z)=\PP(V^{\vee})$. Assume that for all $x \in Z$, both $X$ and $pr_x(X) \seq \PP(W_x^{\vee})$ are linearly normal and non-degenerate. Let $\alpha \neq 0 \in K_{p,1}(S(X),V)$. Then
$$\text{Syz}(\alpha) = \bigcap_{x \in Z} \text{Cone}_x\left( \text{Syz}(pr_x(\alpha)) \right),$$
where, for any $Y \in \PP(W_x^{\vee})$, $Cone_x(Y) \seq \PP(V^{\vee})$ denotes the cone with vertex $x$.
\end{thm}
\begin{proof}
Let $x \in Z \seq \text{Syz}(\alpha)$. By Proposition \ref{AN-proj-syz} $(ii)$, for any $y \in V^{\vee}$ we have a commutative diagram
$$\begin{tikzcd}
V^{\vee} \arrow[r, "\overline{\alpha}"] \arrow[d, "f"]  & \bigwedge^p V \arrow[d, "i_x"] \arrow[r,"i_y"] & \bigwedge^{p-1}V \arrow[d, "-i_x"]   \\
W_x^{\vee} \arrow[r,"\overline{\beta}"] & \arrow[r, "i_{f(y)}"] \bigwedge^{p-1}W_x & \bigwedge^{p-2}W_x,
\end{tikzcd}$$
where $\overline{\alpha}$ resp.\ $\overline{\beta}$ represent $\alpha$ resp.\ $pr_x(\alpha)$ and $f:=pr_x$ is the usual projection map of vector spaces. By Proposition \ref{AN-proj-syz} $(i)$, $y \in \text{Syz}(\alpha)$ if and only if $i_y(\overline{\alpha}(y))=0$. Hence if $y \in \text{Syz}(\alpha)$, then 
$$i_{f(y)} \left( \overline{\beta} (f(y)) \right)=0,$$ 
and hence $pr_x(y)=f(y) \in \text{Syz}(pr_x(\alpha))$, by Proposition \ref{AN-proj-syz} once again. Thus $y \in \text{Cone}\left(\text{Syz}(pr_x(\alpha))\right)$. Hence we have established the inclusion $$\text{Syz}(\alpha) \seq \bigcap_{x \in Z} \text{Cone}_x\left( \text{Syz}(pr_x(\alpha)) \right).$$

For the reverse inclusion, suppose $y \in \bigcap_{x \in Z} \text{Cone}_x\left( \text{Syz}(pr_x(\alpha)) \right)$, or, equivalently $f(y) \in \text{Syz}(pr_x(\alpha))$ for all $x \in Z$. Then $$i_x(i_y(\overline{\alpha}(y)))=0, \text{ for all $x \in Z$}.$$ By Proposition \ref{AN-proj-syz}, we need to show $i_y(\overline{\alpha}(y))=0$. Since $Z$ spans $\PP(V^{\vee})$, it suffices to observe that, for any nonzero $u \in \bigwedge^{p-1}V$
\begin{enumerate}[label=(\roman*)]
\item There is some $z \in V^{\vee}$ such that $i_z(u) \neq 0$.
\item The set of $z \in V^{\vee}$ such that $i_z(u)=0$ forms a subspace.
\end{enumerate}
Indeed, $(ii)$ is obvious, whereas for $(i)$ we observe that $i_{z}(u)=0$ if and only if $u \in \bigwedge^{p-1}W_z$ by \cite{aprodu-nagel}, Remark 1.3. Hence $(i)$ follows from the trivial observation that there exists a codimension one subspace $W \seq V$ with $u \notin \bigwedge^{p-1}W$.
\end{proof}

\section{Extremal Syzygies of Embedded Curves}
In this section we will prove \cite{lin-syz}, Conjecture 0.5, which states that all extremal syzygies of a general curve of non-maximal gonality embedded by a complete linear system of sufficiently high degree arise from a scroll. Let $C$ be a smooth $k$-gonal curve of genus $g \geq 2k-1$, for $k \geq 2$ and $L \in \text{Pic}(C)$ a line bundle. If $C$ is sufficiently general, then there exists a \emph{unique} line bundle $A$ with $\deg(A)=k$, $h^0(A) \geq 2$, \cite{arbarello-cornalba}. Further, for such a general $k$-gonal curve, $h^0(A^2)=3$ or, equivalently, the Brill--Noether locus $W^1_k(C)$ is smooth. Assume further $h^1(L-A)=0$. Consider the embedding $\phi_L: C \hookrightarrow \PP^{r(L)}$ for $r(L):=h^0(L)-1$ and the scroll
$$X_L:=\bigcup_{D \in |A|} \text{Span}(D) \seq \PP^{r(L)},$$
\cite{schreyer1}. The scroll $X_L$ has degree $r(L)+1-k$ in $\PP^{r(L)}$ and has Betti numbers
$$b_{p,1}(X_L,\mathcal{O}_{X_L}(1))=p \binom{r(L)+1-k}{p+1}$$
whereas $b_{p,q}(X_L ,\mathcal{O}_{X_L}(1))=0$ for $q \geq 2$. As seen in the previous section, we have an inclusion
$$res_C \; : \; K_{p,1}(X_L,\mathcal{O}_{X_L}(1)) \hookrightarrow K_{p,1}(C,L).$$
Conjecture 0.6 from \cite{lin-syz} states that if $\deg(L) \geq 2g+k$ then $res_C$ is surjective in the extremal case $p=r(L)-k$, which is the largest value of $p$ such that $b_{p,1}(X_L ,\mathcal{O}_{X_L}(1)) \neq 0$. Note that the surjectivity of $res_C : K_{r(L)-k,1}(X_L,\mathcal{O}_{X_L}(1)) \hookrightarrow K_{r(L)-k,1}(C,L)$ is equivalent to the statement
\begin{equation} \label{inc-syz}
X_L \seq \text{Syz}(\alpha) \; \; \text{for all $\alpha \in K_{r(L)-k,1}(C,L)$}.
\end{equation}
We will prove equation (\ref{inc-syz}) using Proposition \ref{char-syz-proj}, together with the following important result of Eisenbud--Popescu.
\begin{thm} [\cite{eisenbud-popescu}] \label{EP}
Let $X \seq \PP^{r}$ be a rational normal scroll of degree $f$ and $0 \neq \alpha \in K_{f-1,1}(X,\mathcal{O}_X(1))$. Then $\text{Syz}(\alpha) =X$.
\end{thm}
The next proposition will be the key step in our proof.
\begin{prop} \label{1-by-1}
Let $C$ be a general $k$-gonal curve of genus $g \geq 2k-1$ and $k \geq 2$. Suppose $b_{r(L)-k,1}(C,L)=r(L)-k$ for some line bundle $L$ of with $h^1(L-A)=0$. Then
$b_{r(L)+1-k,1}(C,L(x))=r(L)+1-k$ for any $x \in C$.
\end{prop}
\begin{proof}
Set $M=L(x)$ and let $\alpha \in K_{r(M)-k,1}(C,M)$ be nonzero. By assumption $b_{r(M(-x))-k,1}(C,M(-x))$ takes the minimal possible value $r(M(-x))-k$. Hence, by semicontinuity of Koszul cohomology, there exists a dense open $U \seq C$ such that $b_{r(M(-y))-k,1}(C,M(-y))=r(M(-y))-k$ for all $y \in C$, and, further $pr_y(\alpha) \neq 0$, by \cite[Prop.\ 2.14]{aprodu-nagel}. By Proposition \ref{char-syz-proj},
$$\text{Syz}(\alpha)= \bigcap_{y \in U} \text{Cone}_y\left(\text{Syz}(pr_y(\alpha))\right),$$
where $pr_y(\alpha) \in K_{r(M(-y))-k,1}(C,M(-y))$. By assumption the map $$res_C: K_{r(M(-y))-k,1}(X_{M(-y)},\mathcal{O}_{M(-y)}(1)) \to K_{r(M(-y))-k,1}(C,M(-y))$$ is an isomorphism, so Theorem \ref{EP} gives
$$\text{Syz}(pr_y(\alpha))=X_{M(-y)} \seq \PP^{r(M)-1}.$$
Observe that $X_{M(-y)}=pr_y(X_M)$. Hence $X_M \seq \text{Cone}_y(pr_y(X_{M(-y)}))$. Hence $X_M \seq \text{Syz}(\alpha)$, as required.
\end{proof}
We will prove equation (\ref{inc-syz}) by induction. The initial step is provided in the following proposition.
\begin{thm} \label{first-step}
Let $g=2i+1$ and let $C$ be a smooth curve of genus $g$ and gonality $i+1$. Assume there is a unique $A \in W^1_{i+1}(C)$ and, further, for such a line bundle $A$ we have  $h^0(A^2)=3$. Let $L$ be a line bundle of degree $2g$ which is $i$-very ample, or, equivalently, $L-K_C$ is not in the difference variety $C_{i+1}-C_{i-1}$. Then $b_{i,1}(C,L)=i$.
\end{thm}
\begin{proof}
This follows from the results in \cite{generic-secant}, in particular the equality of \emph{cycles} $$\mathfrak{Syz}=\mathfrak{Sec}+i\mathfrak{hur}$$ on $\mathcal{M}_{g,2g}$. Namely, if $D=\sum_{i=1}^{2g} x_i \in |L|$ is a reduced divisor, then the marked curve $(C,D) \in \mathcal{M}_{g,2g}$ is not in the divisor $\mathfrak{Sec}$ by definition. Hence, on an appropriate \'etale cover $\phi: S \to \mathcal{M}_{g,2g}$ about $p=[C,D]$ the order of vanishing of the function defining the divisor $\mathfrak{Syz}(\phi)$ at $p$ is given by $i$ multiplied by the order of vanishing of $\mathfrak{hur}(\phi)$. But the assumption that there is a unique $A \in W^1_{i+1}(C)$ and $h^0(A^2)=3$ shows that $\mathfrak{hur}(\phi)$ vanishes to first-order at $p$ and thus $\mathfrak{Syz}(\phi)$ vanishes to order $i$ so that $b_{i,1}(C,L)\leq i$ by construction of the syzygy divisor $\mathfrak{Syz}$, cf.\ \cite[Thm.\ 3.1]{lin-syz}. The condition $L-K_C \notin C_{i+1}-C_{i-1}$ implies $h^1(L-A)=0$ and that the scroll $X_L$ has degree $i+1$. As we already have seen that $b_{i,1}(C,L) \geq b_{i,1}(X_L,\mathcal{O}_{X_L}(1))= i$, this completes the proof.
\end{proof}
As an immediately corollary, we can prove \cite{lin-syz}, Conjecture 0.6 for any smooth curve $C$ of odd genus $g \geq 5$ and submaximal gonality $k=\frac{g+1}{2}$, assuming there is a unique $A \in W^1_{i+1}(C)$ and, further, $h^0(A^2)=3$.
\begin{cor} \label{g=2k-1}
Let $C$ be a smooth curve of odd genus $g=2i+1$ for $i \geq 3$ and submaximal gonality $i+1$. Assume there is a unique $A \in W^1_{i+1}(C)$ and, further, $h^0(A^2)=3$. Then 
$b_{r(L)-i-1,1}(C,L)=r(L)-i-1$ for all line bundles of degree $\deg(L) \geq 2g+i+1$.
\end{cor}
\begin{proof}
By Proposition \ref{1-by-1}, it suffices to assume $\deg(L)=2g+i+1=5i+3$. By Theorem \ref{first-step}, together with Proposition \ref{1-by-1}, it suffices to show that there is an effective divisor $D \in C_{i+1}$ such that $L-K_C-D \notin C_{i+1}-C_{i-1}$. By \cite{lin-syz}, proof of Theorem 0.2, it suffices to show that the secant variety $V^{2i+1}_{2i+2}(L)$ has dimension at most $i$, and for this it is sufficient to show $W^2_{i+3}(C) = \emptyset$. In the range $i \geq 3$, any $B \in W^2_{i+3}(C)$ would contribute to the Clifford index, so it suffices to show that $\text{Cliff}(C)=i-1$ and, further, $A$ is the unique line bundle of degree at most $g-1$ achieving the Clifford index. But this follows from the well-known result of Hirschowitz--Ramanan that our assumptions imply that Schreyer's Conjecture holds for $C$, see \cite{hirsch} and \cite[Theorem 3.1]{lin-syz}, together with the easy direction of Schreyer's conjecture, \cite[Prop.\ 4.10]{SSW}.
\end{proof}
We now arrive at the main result of this section.
\begin{thm}
Let $C$ be a general $k$-gonal curve $C$ of genus $g \geq 2k-1$ for $k \geq 4$. Let $L$ be an arbitrary line bundle on $C$ of degree $\deg(L) \geq 2g+k$. Then $b_{r(L)-k,1}(C,L)=r(L)-k.$
\end{thm}
\begin{proof}
We mirror the proof of \cite[Thm.\ 0.1]{lin-syz}. Namely, fix $k \geq 4$. We prove the result by induction on the genus $g$ of $C$. If $g=2k-1$, then the claim is Corollary \ref{g=2k-1}. So, assume the claim holds for a general $k$-gonal curve $C$. By Proposition \ref{1-by-1}, it suffices to show that, for a general $k$-gonal curve $X'$ of genus $g+1$ and any line bundle $L'$ on $X'$ of degree $2g+2+k$, we have $b_{g+1,1}(X',L')=g+1$. Using Proposition \ref{1-by-1} once more, it is further sufficient to show that there exists a point $p \in X'$ such that $b_{g,1}(X',L'(-p))$ takes the lowest possible value $g$. Now let $X$ be the nodal curve $C \cup_q E$ where $C$ is a general $k$-gonal curve as above, $q$ is a branch point of some pencil $f: C \to \PP^1$ of degree $k$ and $E$ is an elliptic curve. Then $X$ is a stable, genus $g+1$ curve of compact type, which is a limit of smooth $k$-gonal curves. By semicontinuity of Koszul cohomology, and since $X$ is of compact type, it suffices to show that for any line bundle $L$ on $X$ with
$$ \deg(L_E)=1, \; \; \deg(L_C)=2g+1+k$$
there exists a point $p' \in E\setminus \{q \}$ with
\begin{enumerate} [label=(\roman*)]
\item $b_{g,1}(X,L(-p'))=g$
\item $h^1(X,L(-p'))=h^1(X,L^{\otimes 2}(-2p'))=0$,
\end{enumerate}
see \cite{lin-syz}, proof of Theorem 0.1. Choose a general point $p' \in E\setminus \{q \}$. Statement $(ii)$ follows immediately from the Mayer--Vietoris sequence
$$0 \to L_C(-q) \to L(-p') \to L_E(-p') \to 0.$$ For statement $(i)$, consider the commutative diagram
{\small{$$\xymatrix@C=1em{
\bigwedge^{g+1} H^0(C,L(-q)) \ar[r]^-d \ar[d]^{\alpha}  & \bigwedge^{g} H^0(C,L(-q)) \otimes H^0(C,L(-q)) \ar[r]^-d \ar[d]^{\beta}& \bigwedge^{g-1} H^0(C, L(-q))\otimes H^0(C,L^{\otimes 2}(-2q)) \ar[d]^{\gamma} \\
\bigwedge^{g+1} H^0(X,L(-p')) \ar[r]^-d  & \bigwedge^{g} H^0(X,L(-p')) \otimes H^0(X,L(-p')) \ar[r]^-d &\bigwedge^{g-1} H^0(X, L(-p'))\otimes H^0(X,L^{\otimes 2}(-2p'))
},$$}}
where $\alpha, \beta$ are isomorphisms, and $\gamma$ is induced from the natural composition
$$ H^0(C,L^{\otimes 2}(-2q)) \hookrightarrow H^0(C,L^{\otimes 2}(-q))\cong H^0(X,L^{\otimes 2}(-2p')).$$ Since $\gamma$ is injective, we have a natural isomorphism
$$K_{g,1}(C,L(-q)) \xrightarrow{\sim} K_{g,1}(X,L(-p'))$$
induced on the cohomology of the rows in the above diagram. By the induction hypothesis, $b_{g,1}(C,L(-q))=g$, which completes the proof.
\end{proof}

\section{A converse to Aprodu's Projection Theorem}
Recall the following theorem of Aprodu, \cite{aprodu-higher}.
\begin{thm} [Aprodu's Projection Theorem] \label{aprodu-projection}
Let $C$ be a smooth curve of genus $g$ and suppose $x, y \in C$ are distinct points. Let $D$ be the $g+1$ nodal curve obtained by identifying $x$ and $y$. Suppose $K_{p,1}(C,\omega_C)=0$. Then $K_{p+1,1}(D,\omega_D)=0$.
\end{thm}
One important application of this result is that it provides an approach to proving Green's conjecture for curves of a fixed gonality by induction on the genus. In this sense the most interesting case of the Aprodu Projection Theorem is where $p=g-k+1$, where $k$ is the gonality of $C$. The goal of this section is to prove a partial converse of Aprodu's Theorem, allowing one to deduce Green's Conjecture $K_{g-k+1,1}(C,\omega_C)=0$ by assuming that $D$ satisfies Schreyer's Conjecture on syzygies arising from scrolls, which is a stronger assumption than Green's Conjecture, see \cite{SSW}, \cite{lin-syz}.\medskip

Let $C$ be a smooth curve of genus $g$ and gonality $k \geq 3$, let $A \in W^1_k(C)$ and let $T \in |A|$ be a general divisor. Choose distinct points $x,y \in T$ and let $D$ be the nodal curve of genus $g+1$ obtained by identifying $x$ and $y$. Then there is a base point free line bundle $B$ on $D$ with two sections such that $\nu^* B \simeq A$, where $\nu: C \to D$ is the normalization morphism. \medskip

Embed $D$ in $\PP^g$ via the canonical linear system and let $\pi_p \; : \; \PP^g \dashrightarrow \PP^{g-1}$ be the projection from the node $p \in D$. Then the canonical curve $C \seq \PP^{g-1}$ is the projection $\pi_p(D)$. Further, let $Z \seq \PP^g$ denote the cone over $\pi_p(D)$ with vertex at $p$. Then $D \seq Z$. We denote by $\widetilde{\nu} : \widetilde{Z} \to Z$ the desingularization of $Z$. The strict transform $D'$ of $D$ is isomorphic to $C$ and $\widetilde{\nu}_{|_{D'}} \simeq \nu$.

By \cite[V.2]{hartshorne}, $\widetilde{Z}\simeq \PP(\mathcal{O}_C \oplus \omega_C)$ and $\text{Pic}(\widetilde{Z}) \simeq \mathbb{Z}[\mathcal{H}] \oplus \iota^*\text{Pic}(C)$, where $\mathcal{H}$ denotes the pull-back of the hyperplane section of $\PP^g$ and $$\iota: \PP(\mathcal{O}_C \oplus \omega_C) \to C$$ is the projection map.
\begin{lem} \label{class1}
We have $\mathcal{O}_{\widetilde{Z}}(D') \simeq \mathcal{O}_{\widetilde{Z}}(\mathcal{H}) \otimes \iota^*\mathcal{O}_C(x+y)$.
\end{lem}
\begin{proof}
The strict transform $D'$ corresponds to a section $s : C \to \PP(\mathcal{O}_C \oplus \omega_C)$. We have $s^* \mathcal{H} \simeq \nu^* \omega_D\simeq \omega_C(x+y)$. By Prop.\ 2.6, \cite[V.2]{hartshorne}, we have a short exact sequence
$$0 \to \mathcal{N} \to \mathcal{O}_C \oplus \omega_C \to s^* \mathcal{H} \to 0,$$
with $\iota^*\mathcal{N} \simeq \mathcal{H}(-D') $. By taking determinants of the short exact sequence above
$$\mathcal{N} \simeq \omega_C(-s^*\mathcal{H}) \simeq \mathcal{O}_C(-x-y).$$
Applying $\iota^*$ gives $\mathcal{O}_{\widetilde{Z}}(D') \simeq \mathcal{O}_{\widetilde{Z}}(\mathcal{H}) \otimes \iota^*\mathcal{O}_C(x+y)$, as claimed.
\end{proof}

By the previous lemma, we have a short exact sequence
$$0 \to \mathcal{H}^{ \vee} \left(-\iota^*(x+y) \right) \to \mathcal{O}_{\widetilde{Z}} \to \mathcal{O}_{D'} \to 0$$
which induces an exact sequence of $\mathcal{S}:=\text{Sym}\left(H^0(\widetilde{Z},\mathcal{H})\right)$ modules
$$0 \to \bigoplus_{q \in \mathbb{Z}} H^0\left(\mathcal{H}^{\otimes (q-1)}(-\iota^*(x+y)) \right) \to \bigoplus_{q \in \mathbb{Z}} H^0\left( \mathcal{H}^{\otimes q}\right) \xrightarrow{\alpha} \bigoplus_{q \in \mathbb{Z}} H^0\left( C, \omega_C^{\otimes q}(qx+qy)\right)$$
where we used the identification $C \simeq D'$. 
We let $$\mathbb{M}:=\text{Im}(\alpha)$$
be the image of $\alpha$.
Hence we have a short exact sequence
\begin{align} \label{ses-1}
0 \to \bigoplus_{q \in \mathbb{Z}} H^0\left(\mathcal{H}^{\otimes (q-1)}(-\iota^*(x+y)) \right) \to \bigoplus_{q \in \mathbb{Z}} H^0\left( \mathcal{H}^{\otimes q}\right) \to \mathbb{M} \to 0
\end{align}
of $\mathcal{S}$ modules.

\begin{prop} \label{proj-map}
The following statements hold:
\begin{enumerate}
\item The inclusion $\mathbb{M} \seq \bigoplus_{q \in \mathbb{Z}} H^0\left( C, \omega_C^{\otimes q}(qx+qy))\right)$ induces an isomorphism
$K_{p,1}(\mathbb{M}, \mathcal{S}) \simeq K_{p,1}(C,\omega_C(x+y))$ of syzygy spaces, for all integers $p$.
\item Restricting to the hyperplane section induces an isomorphism $K_{p,q}(\widetilde{Z},\mathcal{H}) \simeq K_{p,q}(C, \omega_C)$, for all $p,q$.
\item We have a natural exact sequence
$$0 \to K_{p,1}(C,\omega_C) \to K_{p,1}(C,\omega_C(x+y)) \xrightarrow{\delta} K_{p-1,1}(\widetilde{Z},-\iota^*(x+y);\mathcal{H}) \to K_{p-1,2}(C, \omega_C)$$
for all $p \in \mathbb{Z}$.
\end{enumerate}
\end{prop}
\begin{proof}
We have the isomorphism $H^0(\widetilde{Z},\mathcal{H}) \simeq H^0(\mathcal{O}_Z(1)) \twoheadrightarrow H^0(\mathcal{H}_{D'}) \simeq H^0(\mathcal{O}_D(1))$, and thus the degree one piece $\mathbb{M}_1$ is isomorphic to $H^0\left( C, \omega_C(x+y))\right)$. The first claim now follows \cite{res-odd}, proof of Lemma 1.3.

Since $C \seq \PP^{g-1}$ is projectively normal, we have surjections $H^0(\widetilde{Z}, \mathcal{H}^{\otimes q}) \simeq H^0(Z,\mathcal{H}^{\otimes q}) \twoheadrightarrow H^0(C,\omega_C^{\otimes q})$ for all $q\geq 0$. The second claim then follows from the proof of the Green--Lefschetz Theorem, \cite[Thm.\ 3.b.7]{green-koszul}.

For the final claim, we have $K_{p,0}(\widetilde{Z},-\iota^*(x+y);\mathcal{H})=0$. Taking Koszul cohomology of the short exact sequence (\ref{ses-1}) yields the exact sequence
$$0 \to K_{p,1}(\widetilde{Z},\mathcal{H}) \to \K_{p,1}(\mathbb{M},\mathcal{S}) \xrightarrow{\delta} K_{p-1,1}(\widetilde{Z},-\iota^*(x+y);\mathcal{H}) \to K_{p-1,2}(\widetilde{Z},\mathcal{H}).$$
The claim then follows from the previous statements.
\end{proof}

We have a short exact sequence
$$0 \to \mathcal{O}_D \to \nu_* \mathcal{O}_C \to \mathcal{O}_p \to 0,$$
inducing injective maps 
$$H^0(D,\omega_D^{\otimes q}) \hookrightarrow H^0(C,\omega_C^{\otimes q}(qx+qy)), \; \; q \geq 0,$$
which are isomorphisms for $q=0,1$. Since we may write
$$K_{p,1}(D, \omega_D) \simeq \text{Ker}\left( \wedge^p H^0(\omega_D) \otimes H^0(\omega_D) \middle / \wedge^{p+1}H^0(\omega_D)  \rightarrow \wedge^{p-1}H^0(\omega_D) \otimes H^0(\omega_D^{\otimes 2}) \right),$$
and likewise for $K_{p,1}(C, \omega_C(x+y))$, we have a canonical isomorphism 
\begin{align} K_{p,1}(C,\omega_C(x+y)) \simeq K_{p,1}(D,\omega_D), \end{align}
for any $p$. \medskip

Our goal is to show that the map $\delta : K_{p,1}(C,\omega_C(x+y)) \to K_{p-1,1}(\widetilde{Z},-\iota^*(x+y);\mathcal{H})$ is injective for $p=g-k+1$ under certain hypotheses, which then implies $K_{g-k+1,1}(C,\omega_C)=0$. To proceed, we need to use the Eagon--Nothcott resolution of a scroll, \cite[\S 1]{schreyer1}. Let
$$\mathcal{E}=\mathcal{O}_{\PP^1}(e_1) \oplus \ldots \oplus \mathcal{O}_{\PP^1}(e_d), \; e_i > 0 \; \text{for all $i$}$$
be a globally generated vector bundle of rank $d$ on $\PP^1$ and let $$X := j(\PP(\mathcal{E})) \seq \PP^r,$$ $r=h^0(\mathcal{O}_{\PP(\mathcal{E})}(1))-1$ be the associated ruled variety of dimension $d$ and minimal degree $r-d+1$, where $j: \PP(\mathcal{E}) \hookrightarrow \PP\left(H^0(\mathcal{O}_{\PP(\mathcal{E})}(1))\right)$ is the natural morphism. \medskip

 Let $p \in X$ be a smooth point and consider the projection $$\pi_p \; : \; \PP^r \dashrightarrow \PP^{r-1}.$$ Let $Y \seq \PP^r$ denote the cone over the image $\pi_p(X)$, with vertex at $p$. Then $X \seq Y$ and the cone $Y$ is a variety of minimal degree of dimension one higher than $X$. We may resolve $Y$ by a smooth rational normal scroll $\widetilde{Y}$, see \cite{eisenbud-harris-minimal}. We have a birational morphism $Bl_p(Y) \to \widetilde{Y}$ from the blow-up of $Y$ at $p$ to $\widetilde{Y}$. Let $\mu: \widetilde{Y} \to Y$ denote the resolution of singularities, and let $X' \seq \widetilde{Y}$ be the strict transform of $X$.  Let $\mathcal{H}$, $\mathcal{R}$ denote the class of the hyperplane and ruling of $\widetilde{Y}$. These classes span $\text{Pic}(\widetilde{Y})$.
\begin{lem} \label{class2}
We have $\mathcal{O}_{\widetilde{Y}}(X') \simeq \mathcal{O}_{\widetilde{Y}}(\mathcal{H}+\mathcal{R})$.
\end{lem}
\begin{proof}
Let $a,b$ be such that $\mathcal{O}_{\widetilde{Y}}(X') \simeq \mathcal{O}_{\widetilde{Y}}(a\mathcal{H}+b\mathcal{R})$.
Firstly, the image of the ruling on the scroll $\PP(\mathcal{E})$ gives the ruling on $\pi_p(X)$ under the map $\pi_p \circ j$, and this pulls back to the ruling $\mathcal{R}$ on $\widetilde{Y}$ under $\pi_p \circ \mu$.  We see from this that $X'$ meets a general ruling $\mathcal{R}$ of $\widetilde{Y}$ in a linear space of codimension one, and so $a=1$. We have $\deg(\widetilde{Y})=\deg(\pi_p(X))=\deg(X)-1$. Thus, if $d=\deg(\widetilde{Y})$, $d+1=\mathcal{H}^{\dim(X)} \cdot (a\mathcal{H}+b\mathcal{R})=ad+b$, and so $b=(1-a)d+1$ which gives $a=1, b=1$.
\end{proof}
By the above lemma, we have a short exact sequence
$$0 \to \mathcal{O}_{\widetilde{Y}}(-\mathcal{H}-\mathcal{R}) \to \mathcal{O}_{\widetilde{Y}} \to \mathcal{O}_{X'} \to 0.$$
Notice that, by the Leray spectral sequence applied to $\widetilde{Y} \to \PP^1$, we have $H^1( \widetilde{Y}, \mathcal{H}^{\otimes q-1}(-\mathcal{R}))=0$ for $q \geq 0$. Thus we have a short exact sequence
$$ 0 \to \bigoplus_{q \in \mathbb{Z}}H^0(\widetilde{Y}, \mathcal{H}^{\otimes q-1}(-\mathcal{R})) \to \bigoplus_{q \in \mathbb{Z}}H^0(\widetilde{Y}, \mathcal{H}^{\otimes q})  \to \bigoplus_{q \in \mathbb{Z}}H^0(X', \mathcal{H}^{\otimes q})  \to 0 $$
of $\text{Sym} \left(H^0(\widetilde{Y},\mathcal{H})\right)$ modules. This induces a long exact sequence
$$0=K_{p,0}\left(\widetilde{Y},-\mathcal{R};\mathcal{H}\right) \to K_{p,1}(\widetilde{Y},\mathcal{H}) \to K_{p,1}(X', \mathcal{H}) \xrightarrow{\Delta} K_{p-1,1}\left(\widetilde{Y},-\mathcal{R};\mathcal{H}\right) \to K_{p-1,2}(\widetilde{Y},\mathcal{H})\to \ldots$$

\begin{lem} \label{proj-scroll}
Set $f=r-d+1=\deg(X)$. Then $$\Delta:K_{f-1,1}(X',\mathcal{H}) \xrightarrow{\sim} K_{f-2,1}(\widetilde{Y}, -\mathcal{R}; \mathcal{H})$$ is an isomorphism for $p=f-1$.
\end{lem}
\begin{proof}
The Eagon--Northcott complex provides minimal resolutions for the scrolls $X'$ and $\widetilde{Y}$, see \cite{EN} and \cite{schreyer1}. These resolutions have length one less than the degree of the variety, and all matrices between terms in the resolution have linear forms as entries, with the exception of the first map, which has quadratic entries. Thus $K_{f-1,1}(\widetilde{Y},\mathcal{H})=K_{f-2,2}(\widetilde{Y},\mathcal{H})=0$.
\end{proof}

We now arrive at the main result of this section. 
\begin{thm} \label{conv-AP}
Let $D$ be the $1$-nodal, $k$-gonal curve as above, with normalization the smooth curve $C$ of genus $g$ and line bundle $B \in W^1_k(D)$ satisfying $\nu^*B \simeq A$, for $k \geq 3$.
Assume \begin{enumerate}
\item $h^0(D,B^{\otimes 2})=3$, and
\item  $b_{g+1-k,1}(D,\omega_D)=g+1-k$. 
\end{enumerate}
Then $K_{g+1-k,1}(C,\omega_C)=0$.
\end{thm}
\begin{proof}
By Proposition \ref{proj-map}, it is equivalent to show that the boundary map
$$\delta \; : \; K_{g+1-k,1}(C,\omega_C(x+y)) \to K_{g-k,1}(\widetilde{Z},-\iota^*(x+y);\mathcal{H})$$
is injective. Consider the scroll $$X=\bigcup_{s \in |B|} \langle s \rangle \seq \PP^{g}$$ induced by the given $g^1_k$ on $D$, \cite[\S 2]{schreyer1}. Then $X$ has degree $f=g+2-k$ and, further, is smooth since $h^0(D,B^{\otimes 2})=3$, \cite[\S 4]{lin-syz}. As above, let $Y$ denote the cone over $\pi_p(X)$ with vertex at $p$, where $p \in D$ is the node and
$$\pi_p \; : \; \PP^g \dashrightarrow \PP^{g-1}$$
is the projection away from $p$. Further, let $\mu \; : \widetilde{Y} \to Y$ denote the resolution of singularities, and let $X' \seq \widetilde{Y}$ denote the strict transform of $X$.\medskip

Let $Z \seq \PP^g$ denote the cone over $C \simeq \pi_p(D)$ with vertex at $p$, and let $\widetilde{\nu} \; : \; \widetilde{Z} \to Z$ be the desingularization. We have a natural diagram
$$\begin{tikzcd}
X' \arrow[r, hook] & \widetilde{Y}\\
C \simeq D' \arrow[u] \arrow[r, hook] & \widetilde{Z} \arrow[u] 
\end{tikzcd}$$
relating the strict transforms $D'$, $X'$ and the desingularizations $\widetilde{Y}$, $\widetilde{Z}$ of the cones over the projections. Pulling back the class of the ruling $\mathcal{R}$ on $\widetilde{Y}$ to $\widetilde{Z}$ yields $\iota^*A$, where $$\iota : \widetilde{Z} \simeq \PP(\mathcal{O}_C \oplus \omega_C) \to C$$ is the projection, as above. 
Let $T$ is the unique element of $|A|$ passing through $x$ and $y$. We have a commutative diagram of short exact sequences
$$\begin{tikzcd}[column sep=1.1em]
0 \arrow[r] & \bigoplus_{q \in \mathbb{Z}}H^0(\widetilde{Y}, \mathcal{H}^{\otimes q-1}(-\mathcal{R})) \arrow[r] \arrow[d, "\alpha \circ r_{\widetilde{Z}}"] & \bigoplus_{q \in \mathbb{Z}} H^0(\widetilde{Y}, \mathcal{H}^{\otimes q})  \arrow[r] \arrow[d, "r_{\widetilde{Z}}"]  & \bigoplus_{q \in \mathbb{Z}}H^0(X', \mathcal{H}^{\otimes q})  \arrow[r] \arrow[d, dashed] & 0 \\
0 \arrow[r] & \bigoplus_{q \in \mathbb{Z}} H^0\left(\widetilde{Z}, \mathcal{H}^{\otimes (q-1)}(-\iota^*(x+y)) \right) \arrow[r] & \bigoplus_{q \in \mathbb{Z}} H^0\left(\widetilde{Z}, \mathcal{H}^{\otimes q}\right) \arrow[r] & \mathbb{M} \arrow[r] & 0
\end{tikzcd}$$
where the maps $r_{\widetilde{Z}}$ are induced by pull-back, where $\alpha$ is induced by multiplication by $\iota^* T(-x-y)$, where the lower row is the short exact sequence  (\ref{ses-1}) defined earlier and where the dashed arrow is the induced map. Note that, by Lemmas \ref{class1} and \ref{class2}, the pull-back of the divisor $X'$ to $\widetilde{Z}$  is $D'\cup \iota^{-1}(T(-x-y))$.

 Upon taking Koszul cohomology, we get the commutative diagram
$$\begin{tikzcd}
K_{g+1-k,1}(X',\mathcal{H}) \arrow[r, "\Delta "] \arrow[d, "r_{D'} "]  & K_{g-k,1}(\widetilde{Y},-\mathcal{R}; \mathcal{H}) \arrow[d, "r_{\widetilde{Z}} "]  \\
K_{g+1-k,1}(C,\omega_C(x+y)) \arrow[dr, bend right, "\delta"]  & \arrow[d, "\alpha"] K_{g-k,1}(\widetilde{Z},-\iota^*A; \mathcal{H})\\
 & K_{g-k,1}(\widetilde{Z},-\iota^*(x+y);\mathcal{H})
\end{tikzcd}$$
where $\Delta$ is an isomorphism by Proposition \ref{proj-scroll}. It thus suffices to show that $r_{D'}$ is an isomorphism, whereas $\alpha$ and $r_{\widetilde{Z}}$ are injective. \medskip

To see that $r_{D'}$ is an isomorphism, first note that, since we have birational morphisms $Bl_p X \to X' \to X$, we have isomorphisms $H^0(Bl_pX,q\mathcal{H}) \simeq H^0(X',q\mathcal{H}) \simeq H^0(X,q\mathcal{H})$ for all $q$ and so $$K_{p,1}(X',\mathcal{H}) \simeq K_{p,1}(X,\mathcal{H}).$$ Next, $r_{D'}$  is injective by the same proof as \cite{lin-syz}, Lemma 4.4 (note $D' \seq Bl_p X$). From the Eagon--Northcott resolution, we see $b_{g+1-k,1}(X,\mathcal{H})=f-1=g+1-k$. Hence the claim follows from the assumption $b_{g+1-k,1}(D,\omega_D)=g+1-k$, since $K_{p,1}(D,\omega_D) \simeq K_{p,1}(C,\omega_C(x+y))$ for all $p$.\medskip

The injectivity of $\alpha$ follows from the commutative diagram
$$\begin{tikzcd} [column sep=1.1em]
\wedge^{g-k}H^0(\mathcal{H}) \otimes H^0(\mathcal{H}-\iota^*A) \arrow[r, "\gamma_1"] \arrow[d, hook] & \arrow[d, hook] \wedge^{g-k-1}H^0(\mathcal{H}) \otimes H^0(2\mathcal{H}-\iota^*A) \\
\wedge^{g-k}H^0(\mathcal{H}) \otimes H^0(\mathcal{H}-\iota^*(x+y)) \arrow[r, "\gamma_2"]  & \wedge^{g-k-1}H^0(\mathcal{H}) \otimes H^0(2\mathcal{H}-\iota^*(x+y)),
\end{tikzcd}$$
since $\alpha$ is the induced map $\text{Ker}(\gamma_1) \hookrightarrow \text{Ker}(\gamma_2)$.\medskip

It remains to show $r_{\widetilde{Z}}$ is injective. To begin, the commutative diagram
$$\begin{tikzcd}
H^0(\widetilde{Y},\mathcal{H}-\mathcal{R}) \arrow[r, hook, "\otimes s"] \ar[d, "r_{\widetilde{Z}}"] &  \arrow[d, "\simeq"]H^0(\widetilde{Y}, \mathcal{H}) \\
H^0(\widetilde{Z},\mathcal{H}-\iota^*A) \arrow[r, hook, "\otimes s_{\widetilde{Z}}"] & H^0(\widetilde{Z}, \mathcal{H})
\end{tikzcd}$$
for general $s \in |\mathcal{R}|$ shows that the pull-back map $r_{\widetilde{Z}} : H^0(\widetilde{Y},\mathcal{H}-\mathcal{R}) \hookrightarrow H^0(\widetilde{Z},\mathcal{H}-\iota^*A)$ is injective. Next, the Eagon--Northcott resolution of $$\Gamma_{\widetilde{Y}}(\mathcal{H}; -\mathcal{R}) :=\bigoplus_{q \in \mathbb{Z}} H^0(\widetilde{Y},q\mathcal{H}-\mathcal{R}),$$
as described in \cite[Pg.\ 111]{schreyer1}, is $1$-linear, i.e.\ $K_{p,q}(\widetilde{Y}, -\mathcal{R}; \mathcal{H})=0$ unless $q=1$ for all $p \geq 0$. In particular, $K_{p,2}(\widetilde{Y}, -\mathcal{R}; \mathcal{H})=0$ for all $p \geq 0$. Letting $M_{\mathcal{H}}$ be the kernel bundle
$$0 \to M_{\mathcal{H}} \to H^0(\mathcal{H}) \otimes \mathcal{O}_{\widetilde{Y}} \to \mathcal{H}_{\widetilde{Y}} \to 0,$$
then this implies $H^1(\widetilde{Y}, \wedge^p M_{\mathcal{H}} (\mathcal{H}-\mathcal{R}))=0$ for all $p \geq 0$, by the kernel bundle description of Koszul cohomology, \cite{aprodu-nagel} (note that we have already observed $H^1(\widetilde{Y},\mathcal{H}-\mathcal{R})=0$). For all $p \geq 0$ we have the short exact sequence
$$0 \to \bigwedge^p M_{\mathcal{H}}\otimes (\mathcal{H}-\mathcal{R}) \to \bigwedge^p H^0(\widetilde{Y},\mathcal{H})\otimes (\mathcal{H}-\mathcal{R}) \to \bigwedge^{p-1}M_{\mathcal{H}} \otimes (2\mathcal{H}-\mathcal{R}) \to 0$$
and canonical isomorphisms 
\begin{align*}
K_{p,1}(\widetilde{Y}, -\mathcal{R} ; \mathcal{H}) & \simeq H^0(\widetilde{Y}, \bigwedge^p M_{\mathcal{H}} \otimes (\mathcal{H}-\mathcal{R})) \\
K_{p,1}(\widetilde{Z}, -\iota^*A ; \mathcal{H}) & \simeq H^0(\widetilde{Z}, \bigwedge^p M_{\mathcal{H}} \otimes (\mathcal{H}-\iota^*A)).
\end{align*}
There is a commutative diagram 
$$\begin{tikzcd}[column sep=1.1em]
0 \ar[r] &  \ar[r]  \ar[d, "r_{\widetilde{Z}}"] H^0(\widetilde{Y}, \wedge^p M_{\mathcal{H}} (\mathcal{H}-\mathcal{R})) & \ar[r] \ar[d, hook] \wedge^p H^0(\mathcal{H})\otimes H^0(\widetilde{Y}, \mathcal{H}-\mathcal{R}) & \ar[r] \ar[d] H^0(\widetilde{Y}, \wedge^{p-1}M_{\mathcal{H}} (2\mathcal{H}-\mathcal{R})) & 0 \\
0 \ar[r] &  \ar[r] H^0(\widetilde{Z}, \wedge^p {M_{\mathcal{H}}} (\mathcal{H}-\iota^*A)) & \ar[r] \wedge^p H^0(\mathcal{H})\otimes H^0(\widetilde{Z},\mathcal{H}-\iota^*A) & H^0(\widetilde{Z}, \wedge^{p-1}M_{\mathcal{H}} (2\mathcal{H}-\iota^*A))
\end{tikzcd}$$
with exact rows. The claim now follows from the snake lemma.

\end{proof}

We end this section by proving a result on Green's Conjecture for curves of even genus and maximal gonality.
\begin{thm}
Let $C$ be a smooth curve of genus $g=2n$ and gonality $k=n+1$. Suppose for $x, y \in C$ general, there is at most one $A \in W^1_{n+1}(C)$ such that $A(-x-y)$ is effective. Further assume $h^0(C,A^{\otimes 2})=3$. Then $C$ satisfies Green's Conjecture.
\end{thm}
\begin{proof}
If, for $x, y \in C$ general, there is no $A \in W^1_{n+1}(C)$ with $A(-x-y)$ effective, then $\dim W^1_{n+1}(C)=0$ and the result follows from \cite{aprodu-remarks}. So we may assume that there is precisely one $A \in W^1_{n+1}(C)$ such that $A(-x-y)$ is effective. Let $D$ be the $1$-nodal curve of genus $2n+1$ obtained by identifying $x$ and $y$ and let $\nu: C \to D$ be the normalization. There is a line bundle $B$ on $D$ of degree $n+1$ and with $h^0(D,B)=2$ such that $\nu^*B\simeq A$. Further, $h^0(C,A^{\otimes 2})=3$ implies that $h^0(D,B^{\otimes 2})=3$ (we have $h^0(D,B^{\otimes 2})\geq 3$ by the base-point free pencil trick). Note further that if $T \in \overline{\text{Jac}}(D)$ is a rank one, torsion free sheaf of degree $n+1$ with $h^0(D,T) \geq 2$, then we must have $T \simeq B$. Indeed, such a $T$ must be locally free, or else $T \simeq \nu_* T'$ for $T' \in W^1_n(C)$, contradicting that $C$ has gonality $n$, and then $T$ must be $B$ from the assumption that there is at most one $A \in W^1_{n+1}(C)$ such that $A(-x-y)$ is effective. By \cite[Remark 3.2]{lin-syz}, $b_{n,1}(D,\omega_D)=n$. From Theorem \ref{conv-AP} we deduce that $C$ satisfies Green's Conjecture.
\end{proof}

\section{Green's Conjecture for Elliptic Covers}
In this section we prove Green's conjecture for general elliptic covers. We first fix notation. Let $E$ denote a smooth elliptic curve, with polarization $L=\mathcal{O}_E(p)$, for $p \in E$ a fixed point. We denote by $\mathcal{M}_{g}(E,d_1)$ the moduli space of stable maps $f: C \to E$, with $C$ smooth of genus $g \geq 2$ and $\deg(f^*L)=d_1 \geq 1$, and likewise let $\mathcal{M}_{g}(\PP^1,d_2)$ denote the moduli space of degree $d_2 \geq 1$, genus $g$ stable maps to $\PP^1$.  We lastly write
$$\mathcal{M}_{g}(E \times \PP^1,d_1,d_2)$$
for the Deligne--Mumford stack of stable maps $f: C \to E \times \PP^1$ with $\deg(f^*L)=d_1, \deg f^* \mathcal{O}_{\PP^1}(1)=d_2$. 

\begin{lem}
Fix $g \geq 2$. The stacks $\mathcal{M}_{g}(E,d_1)$ respectively $\mathcal{M}_{g}(\PP^1,d_2)$ are smooth of dimensions $2g-2$ resp.\ $2g-2+2d_2$.
\end{lem}
\begin{proof}
For any smooth, projective variety $X$ and non-constant stable map $f: C \to X$ from a smooth curve $C$, the deformation theory of $f$ is determined by the normal sheaf $N_f$, \cite[\S 4]{bogomolov-hassett-tschinkel} which fits into the short exact sequence
$$0 \to T_C \xrightarrow{df} f^*T_X \to N_f \to 0.$$
In particular, when $\dim X=1$, $N_f$ is supported on the ramification locus of $f$, and so $h^1(C,N_f)=0$. Thus $\mathcal{M}_{g}(E,d_1)$ and $\mathcal{M}_{g}(\PP^1,d_2)$ are smooth, with dimension at a point $f: C \to X$ given by $h^0(C,N_f)$, where $X \in \{E,\PP^1 \}$. For $X=E$, we have $$h^0(C,N_f)=\chi(\mathcal{O}_C)-\chi(\omega_C^*)=2g-2,$$
whereas for $X=\PP^1$,
$$h^0(C,N_f)=\chi(f^*\mathcal{O}_{\PP^1}(2))-\chi(\omega_C^*)=2g-2+2d_2.$$
\end{proof}

\begin{prop} \label{dim-ct-1}
With notation as above, let $[f: C \to E \times \PP^1] \in \mathcal{M}_{g}(E \times \PP^1,d_1, d_2)$, for $d_1, d_2 \geq 1$ be a point such that $f$ is birational to its image. Then each component of $\mathcal{M}_{g}(E \times \PP^1,d_1, d_2)$ containing $[f]$ is generically reduced and has dimension $g-1+2d_2$.
\end{prop}
\begin{proof}
We follow \cite{arbarello-cornalba}. First of all, observe that each component $I$ of $\mathcal{M}_{g}(E \times \PP^1,d_1, d_2)$ containing $[f]$ has dimension at least
$$\dim \mathcal{M}_{g}(E,d_1)+\mathcal{M}_{g}(\PP^1,d_2)-\dim \mathcal{M}_g=g-1+2d_2.$$
Let $[h: C' \to E \times \PP^1] \in I$ be a general point. The normal sheaf $N_h$ of the morphism $h$ fits into an exact sequence
$$0 \to K_h \to N_h \to N'_h \to 0,$$
of sheaves on $C'$, where $K_{h}$ is (non-canonically) isomorphic to $\mathcal{O}_Z$, where $Z$ is the ramification locus of $h$, and where $N'_h$ is a line bundle. 
By \cite[Lemma 1.4]{arbarello-cornalba}, $$h^0(N'_h) \geq g-1+2d_2 \geq g+1.$$
For any line bundle $L$ on $C'$, if $h^1(L) \neq 0$ then $|L|$ is a sublinear system of $|\omega_{C'}|$ and hence $h^0(L) \leq g$. Thus $h^1(N'_h)=0$ and hence
$h^1(N_h)=0$, so that $I$ is smooth at $[h]$. Applying \cite[Lemma 1.4]{arbarello-cornalba} again, we may now conclude that $K_h=0$, $h$ is unramified and $N_h$ is locally free of degree $$ \deg(h^*T_{E \times \PP^1})+2g-2=2g-2+2d_2.$$
Thus $\dim I=h^0(N_h)=\chi(N_h)=g-1+2d_2$ as required.
\end{proof}
We denote by $\mathcal{H}^E_g(d_1) \seq \mathcal{M}_g(E,d_1)$ the open locus of \emph{primitive} covers with simple ramification. The space $\mathcal{H}^E_g(d_1)$ is then nonempty and irreducible for $g \geq 2$, by a result of Gabai--Kazez, \cite{gabai-kazez} (see also \cite{bujokas}). 

For a smooth curve $C$, let $W^1_k(C)$ be the Brill--Noether variety of line bundles $L$ of degree $k$ with at least two sections, and let $G^1_d(C)$ be the variety of $g^1_d$'s, i.e.\ pairs of a line bundle $A \in W^1_k(C)$ together with a base-point free linear system $V \seq H^0(A)$ of dimension two, up to the natural $\text{PGL}(2)$ action, \cite{ACGH1}. If $C$ has gonality $k$, then elements of $W^1_k(C)=G^1_k(C)$ are called \emph{minimal pencils}.
\begin{cor} \label{BN-estimates}
Let $[f: C \to E] \in \mathcal{H}^E_g(d_1)$ be a general point, let $(A,V) \in G^1_{d_2}(C)$ and suppose $A$ is not isomorphic to $f^*B$ for some $B \in \text{Pic}(E)$ with $\deg(B) \geq 2$. Then $$\dim_{[A]}G^1_{d_2}(C)=2d_2-g-2.$$ In particular, if $2d_1 \leq \lfloor \frac{g+3}{2} \rfloor $, then $C$ has gonality $2d_1$ and if $2d_1 > \lfloor \frac{g+3}{2} \rfloor $ then $C$ has gonality $\lfloor \frac{g+3}{2} \rfloor$. Moreover, if $2d_1 < \lfloor \frac{g+3}{2} \rfloor $, all minimal pencils of $C$ are of the form $f^*B$, $B \in \text{Pic}^2(E)$.
\end{cor}
\begin{proof}
Let $(A,V) \in G^1_{d_2}(C)$ be base-point free and suppose $A$ is not the pull-back of a line bundle from $E$ of degree at least two. Then $V$ induces a map $C \to \PP^1$ of degree $d_2$ and we let $[h: C \to E \times \PP^1] \in \mathcal{M}_{g}(E \times \PP^1,d_1, d_2)$ be the product of this map with $f$. We claim that $h$ is birational to its image. Indeed, otherwise let $D=h(C)$ be the image of $h$. Since $f$ is simply ramified and primitive projection to the first factor must induce an isomorphism $pr_1: D \xrightarrow{\sim} E$. But we must then have that $A$ is the pull-back of a line bundle $B$ from $E$ with $\deg(B) \geq 2$, which is a contradiction. So $h$ is birational. Since we are assuming $[f]$ is general, each component of $\mathcal{M}_{g}(E \times \PP^1,d_1, d_2)$ containing $[h]$ dominates $\mathcal{H}^E_g(d_1)$ under the natural forgetful morphism and all fibres have dimension $$g-1+2d_2-(2g-2),$$
by Proposition \ref{dim-ct-1}. After subtracting $3=\dim \text{PGL}(2)$, we see that each component of $G^1_{d_2}(C)$ containing $[A]$ has dimension $2d_2-g-2$, as required. In particular, we must have $d_2 \geq \frac{g+2}{2}$. In fact, since $d_2$ is an integer we have $d_2 \geq \lfloor \frac{g+3}{2} \rfloor $ (which is the gonality of a general curve of genus $g$). The remaining statements follow immediately. 
\end{proof}

We can now prove the main result of this section.
\begin{thm}
Let $[f: C \to E] \in \mathcal{H}^E_g(d_1)$ be a general point. Then Green's Conjecture holds for $C$.
\end{thm}
\begin{proof}
We may assume $d_1 \geq 2$ as Green's conjecture holds for all elliptic curves. If $2d_1 \geq \frac{g+3}{2}$, then by Corollary \ref{BN-estimates}, $C$ has maximal gonality $k:=\lfloor \frac{g+3}{2} \rfloor$ and, further, $\dim G^1_{k+n}(C) \leq n$ for $n \leq g+2-2k$. Thus the statement follows from a theorem of Hirschowitz--Ramanan \cite{hirsch} combined with Voisin's Theorem, \cite{V1}, \cite{V2} (in the odd genus case) and Aprodu, \cite[Theorem 2]{aprodu-remarks} (in even genus).

So it suffices to prove that, for fixed $d_1 \geq 2$ and all $g \geq 4d_1-3$, the general point $[f: C \to E] \in \mathcal{H}^E_g(d_1)$ satisfies $b_{g-2d_1+1,1}(C,\omega_C)=0$, which further forces $\text{gon}(C)=2d_1=\text{Cliff}(C)+2$.
We will prove this vanishing by induction, with the base case $g=4d_1-3$ holding by the above. So suppose $[f: C \to E] \in \mathcal{H}^E_g(d_1)$ is a general point, with $g \geq 4d_1-3$, and suppose $b_{g-2d_1+1,1}(C,\omega_C)=0$. Let $x, y \in C$ be distinct points such that $f(x)=f(y)$ and let $D$ be the curve of genus $g+1$ obtained by identifying $x$ and $y$. Let $g: D \to E$ be the unique morphism factoring through $f: C \to E$. By the Aprodu Projection Theorem \ref{aprodu-projection}, $b_{g+2-2d_1,1}(D,\omega_D)=0$, which implies $b_{g+2-2d_1,1}(D',\omega_D')=0$ for a general point $[f': D' \to E] \in \mathcal{H}^E_{g+1}(d_1)$ by semicontinuity of Koszul cohomology.
\end{proof}

\end{document}